\documentclass[10pt]{article}
\usepackage[utf8]{inputenc}

\usepackage{amssymb,amsthm,amsmath}
\usepackage{mathrsfs}
\usepackage{enumerate}
\usepackage{graphicx,xcolor}
\usepackage{setspace}
\usepackage[hidelinks]{hyperref}

\newcommand{\dd}{\mathrm{d}}
\newcommand{\E}{\mathbb{E}}

\newcommand{\R}{\mathbb{R}}
 
\newcommand{\Z}{\mathbb{Z}}
\newcommand{\e}{\varepsilon}
\newcommand{\p}[1]{\mathbb{P}\left( #1 \right)}
\newcommand{\scal}[2]{\left\langle #1, #2 \right\rangle}

\DeclareMathOperator{\sgn}{sgn}
\DeclareMathOperator{\vol}{vol}

\usepackage{xpatch}
\makeatletter
\def\thm@space@setup{%
  \thm@preskip=12pt plus 0pt minus 0pt
  \thm@postskip=0pt plus 0pt minus 0pt
}
\xpatchcmd{\proof}{6\p@\@plus6\p@\relax}{\z@skip}{}{}
\makeatother

\newtheorem{theorem}{Theorem}
\newtheorem{lemma}[theorem]{Lemma}
\newtheorem{corollary}[theorem]{Corollary}

\theoremstyle{remark}
\newtheorem{remark}[theorem]{Remark}

\theoremstyle{definition}

\title{Convexity properties of sections of $1$-symmetric bodies and Rademacher sums}

\author{Joseph Kalarickal, David Rotunno, Salil Singh\footnote{Email: salils@andrew.cmu.edu}, Tomasz Tkocz\footnote{Research supported in part by the NSF grant DMS-2246484.}}

\date{\begin{normalsize}
\emph{Carnegie Mellon University, Pittsburgh, PA 15213, USA}
\end{normalsize}}

\begin{document}

\maketitle

\begin{abstract} 
We establish a monotonicity-type property of volume of central hyperplane sections of $1$-symmetric convex bodies, with applications to chessboard cutting. We parallel this for projections with a new convexity-type property for Rademacher sums.
\end{abstract}

\bigskip

\begin{footnotesize}

\end{footnotesize}

\bigskip

\section{Introduction and results}

A line can intersect at most $2N-1$ squares of the standard $N \times N$ chessboard and this is achieved by a diagonal line pushed down a bit. It is only  recently that this fact has been generalised to higher dimensions and arbitrary convex bodies. Specifically, given a convex body $K$ in $\R^n$ and $N \geq 1$, consider the (open) cells of the lattice $\frac{1}{N}\Z^n$, that is the cubes $z + (0,\frac1N)^n$, $z \in \frac{1}{N}\Z^n$, and let $C_{K}(N)$ be the maximal number of cells contained in $K$ that a hyperplane in $\R^n$ can intersect. For the standard cube, $K = [0,1]^n$, we simply write $C_n(N) = C_{[0,1]^n}(N)$, so $C_2(N) = 2N-1$. B\'ar\'any and Frankl in \cite{FB1} showed that $C_3(N) \leq \frac{9}{4}N^2 + 2N+1$ for all $N \geq 1$ and $C_3(N) \geq \frac{9}{4}N^2 + N -5$ for all $N$ sufficiently large. In the companion work \cite{FB2}, they established the exact asymptotics of $C_{K}(N)$ as $N \to \infty$ for a fixed body $K$. Their main result is that
\[
C_{K}(N) = \beta_KN^{n-1}(1+o(1)), \qquad N \to \infty,
\]
with the constant $\beta_K$ of the leading term given by
\begin{equation}\label{eq:VK-def}
\beta_K = \max_{a \in \R^n\setminus \{0\}}\max_{t \in \R} \frac{\|a\|_1}{|a|}\vol_{n-1}(K \cap (ta + a^\perp)).
\end{equation}
Here and throughout, $|x|$ is the standard Euclidean norm, whereas $\|x\|_p$ is the $\ell_p$ norm of a vector $x$ in $\R^n$, so $|x| = \|x\|_2$. It is a consequence of the Brunn-Minkowski inequality that when $K$ is symmetric, say about the origin, then given an outer-normal vector $a$, the maximal volume section $\vol_{n-1}(K \cap (ta + a^\perp))$ is the central one at $t = 0$. Thus we define the $0$-homogeneous function,
\begin{equation}\label{eq:VKa-def}
V_K(a) = \frac{\|a\|_1}{|a|}\vol_{n-1}(K \cap a^\perp), \qquad a \in \R^n \setminus\{0\}
\end{equation}
and for origin-symmetric $K$, we have $\beta_K = \max_a V_K(a)$.

B\'ar\'any and Frankl in \cite{FB2} conjectured that for the unit cube $Q_n = [-\frac12,\frac12]^n$, the maximum of $V_{Q_n}$ is attained at diagonal vectors. This was confirmed by Aliev in \cite{A1}, 
\[
\beta_{Q_n} = V_{Q_n}((1,\dots,1)).
\]

We refine this result to a Schur-convexity statement. Let us briefly recall the notion. Throughout, $\R_+$ denotes the nonnegative real numbers. Given $a\in \R_+^n$, we denote $a^* = (a_1^*, \dots , a_n^*)$ to be the nonincreasing rearrangement of $a$, i.e.  $a_1^* \geq a_2^* \geq \dots \geq  a_n^*.$ For $a, b \in \R_+^n$, we say $a \prec b$, that is, $a$ is \textit{majorized} by $b$, if they have the same total mass, $\sum_{i=1}^n a_i = \sum_{i=1}^n b_i$ and for every $k \leq n-1$,
we have $\sum_{i=1}^k a_i^* \leq \sum_{i=1}^k b_i^*$. 
In particular, on the probability simplex, i.e. for every vector $x = (x_1, \dots, x_n) \in \R_+^n$ with $\sum_{i=1}^n x_i = 1$, we have
\[
\left(\frac{1}{n} , \dots, \frac{1}{n} \right) \prec x \prec (1, 0, 0, \dots , 0).
\]

In information theoretic language, $x \prec y$ should be interpreted to mean $x$ is a more chaotic or entropic probability mass distribution than $y$. A map $\varphi\colon D \to \R$ is said to be Schur-concave (resp. convex) on a subset $D$ of $\R_+^n$ if it is an order reversing (resp. preserving) map, i.e.: $x \prec y$ implies that $\varphi (x) \geq \varphi(y)$ (resp.$\varphi (x) \leq \varphi(y)$). On the probability simplex, a classic example would be the Shannon entropy $H(p_1, \dots, p_n) = -\sum_{i=1}^n p_i \log(p_i)$ which is maximized at the most chaotic state $\left(\frac{1}{n}, \dots, \frac{1}{n} \right)$.

For a comprehensive exposition on majorization and Schur-convexity, we refer for instance to Chapter II of Bhatia's monograph \cite{bhatia}.

In fact, not only does this refinement hold for the cube, but for all $1$-symmetric convex bodies. A~convex body $K$ in $\R^n$ is called $1$-symmetric if it is symmetric with respect to every coordinate hyperplane $\{x \in \R^n, x_j=0\}$, $j \leq n$, and $K$ is invariant under permutations of the coordinates.

\begin{theorem}\label{thm:Schur}
Let $K$ be a $1$-symmetric convex body in $\R^n$. Then the function $a \mapsto V_K(a)$ defined in \eqref{eq:VKa-def} is Schur concave on $\R_+^n$. In particular, for the chessboard cutting constant defined in \eqref{eq:VK-def}, we have $\beta_K =  \sqrt{n}\vol_{n-1}(K \cap (1,\dots,1)^\perp)$.
\end{theorem}

Our short proof crucially relies on Busemann's theorem from \cite{Bus}, combined with the symmetries of the body. In contrast, Aliev's approach from \cite{A1} employs Busemann's theorem in a further geometric argument on the plane which did not seem to allow for the present generalisation to Schur-convexity. We record Busemann's theorem for future use.

\begin{theorem}[Busemann, \cite{Bus}]\label{thm:Bus}
Let $K$ be an origin-symmetric convex body in $\R^n$. Then the function
\begin{equation}\label{eq:def-N}
N_K(x) =  \frac{|x|}{\vol_{n-1}(K \cap x^\perp)}, \qquad x \neq 0
\end{equation}
extended at $0$ by $0$ defines a norm on $\R^n$.
\end{theorem}

We also refer to Theorem 3.9 in \cite{MP} for a generalisation to lower-dimensional sections, as well as to Theorem 5 in \cite{Ball} for an extension to log-concave functions. 

With Busemann's theorem in hand, we can motivate our next result. Hyperplane sections of the  unit volume cube $Q_n = [-\frac12, \frac12]^n$ admit a curious probabilistic formula: if we let $\xi_1, \xi_2, \dots$ be i.i.d. random vectors uniform on the unit sphere $S^2$ in $\R^3$, then for a \emph{unit} vector $a \in \R^n$, we have
\[
\vol_{n-1}(Q_n \cap a^\perp) = \E\left[\left|a_1\xi_1+\dots + a_n\xi_n\right|^{-1}\right],
\]
see \cite{KK}. Thus Busemann's theorem in particular asserts that the function
\[
x \mapsto \frac{|x|}{\vol_{n-1}(Q_n \cap x^\perp)} = \frac{|x|}{\E\left[\left|\sum_{j=1}^n \frac{x_j}{|x|}\xi_j\right|^{-1}\right]} = \left(\E\left|\sum_{j=1}^n x_j\xi_j\right|^{-1}\right)^{-1}
\]
is convex on $\R^n$. A perhaps much simpler (geometrically dual) analogue of this fact is that for i.i.d. Rademacher random variables $\varepsilon_1, \varepsilon_2, \dots$ (random signs, $\p{\varepsilon_j = \pm 1} = \frac12$), the function
\[
x \mapsto \E\left| x_1\varepsilon_1 + \dots + x_n\varepsilon_n\right|
\]
is plainly convex on $\R^n$. The duality we loosely allude to here is that Busemann's theorem refers to volume of sections of the cube, whereas the quantity $\E\left| x_1\varepsilon_1 + \dots + x_n\varepsilon_n\right|$ is proportional to the volume of the orthogonal projection of the cross-polytope onto $x^\perp$ (see, e.g. the remark after (3) in \cite{Ball-wave}).

Resisting great efforts and prompting significant activity across geometric functional analysis, the conjectured logarithmic Brunn-Minkowski inequality posed in \cite{BLYZ} can be equivalently stated as a convexity property of sections of the cube (see \cite{NT}), which in particular would imply that the function
\begin{equation}\label{eq:logBM-xi}
t \mapsto -\log\left(\E\left|e^{t_1}\xi_1 + \dots + e^{t_n}\xi_n\right|^{-1}\right)
\end{equation}
is convex on $\R^n$. In light of the probabilistic formulation we note above, one can view the quantity within the arguments of the logarithm as the volume of the section of a box that is the result of coordinate-wise rescaling of a unit cube by exponential weights depending on $t$. Thus the convexity of \eqref{eq:logBM-xi} amounts to the log-concavity of hyperplane sections of such boxes. Geometrically, the conjectured logarithmic Brunn-Minkowski inequality in $\R^n$ would hold true for two bodies which are symmetric polytopes with at most $2n+2$ facets with the same sets of outer-normal vectors.

To the best of our knowledge, even this apparent ``toy-case'' remains unproved. Driven by the analogy with random signs, we establish the following result.

\begin{theorem}\label{thm:radem-log-convex-p}
Let $\varepsilon_1, \varepsilon_2, \dots$ be independent Rademacher random variables. For every $n \geq 1$ and $p \geq 1$, the function
\[
\Phi(t_1, \dots, t_n) = \log\E\left|e^{t_1}\e_1 + \dots + e^{t_n}\e_n\right|^p
\]
is convex on $\R^n$.
\end{theorem}

Our proof leverages the usual H\"older duality, but nontrivially restricted to random variables having nonnegative correlations with the random signs.

We present the proofs in the next section. The final section is devoted to further remarks. In particular, with the same method, we obtain an extension of Aliev's result from \cite{A2}. We make precise the alluded geometric duality and a connection of our Theorem \ref{thm:radem-log-convex-p} to Saroglou's result from \cite{Sar2}.

\paragraph{Acknowledgements.} 
We should very much like to thank an anonymous referee for their careful reading of the manuscript and helpful suggestions.

\section{Proofs}

\subsection{Proof of Theorem \ref{thm:Schur}}

First note that by the symmetries of $K$, function $V_K$ is also symmetric (under permuting the coordinates of the input), as well as unconditional, that is $V_K(a_1, \dots, a_n) = V_K(|a_1|,\dots, |a_n|)$. Fix $x, y \in \R_+^n$ such that $x \prec y$, that is $y$ majorizes $x$. In particular, $\|x\|_1 = \|y\|_1$. Thus, to show $V_K(x) \geq V_K(y)$, equivalently, we would like to show that
\[
\frac{1}{|x|}\vol_{n-1}(K \cap x^\perp) \geq \frac{1}{|y|}\vol_{n-1}(K \cap y^\perp),
\]
that is $N(x) \leq N(y)$ with
\[
N(a) = \frac{|a|}{\vol_{n-1}(K \cap a^\perp)}.
\]
By Theorem \ref{thm:Bus}, $N$ is convex. By the symmetries of $K$, function $N$ is symmetric. Since $x \prec y$, we have  $x = \sum_{\sigma} \lambda_{\sigma} y_{\sigma}$ for some nonnegative weights $\lambda_{\sigma}$ adding up to $1$, where the sum is over all permutations $\sigma$ of $\{1, \dots, n\}$ and $y_{\sigma} = (y_{\sigma(1)}, \dots, y_{\sigma(n)})$. By the convexity of $N$ and its symmetry,
\[
N(x) = N\left(\textstyle\sum \lambda_\sigma y_\sigma\right)  \leq \sum \lambda_\sigma N(y_\sigma) = \left(\sum \lambda_\sigma\right)N(y) = N(y).
\]
This finishes the proof.$\hfill\square$

\subsection{Proof of Theorem \ref{thm:radem-log-convex-p}}

Fix $p \geq 1$ and let $q \in [1,\infty]$ be its conjugate, $\frac{1}{p} + \frac{1}{q} = 1$. Let $B_q$ be the closed unit ball in $L_q$ (of the underlying probability space with the norm $\|Y\|_q = (\E|Y|^q)^{1/q}$). For $t \in \R^n$, we denote
\[
X_t = \sum_j e^{t_j}\e_j.
\]
Thanks to H\"older's inequality, we have
\[
\left\|X_t\right\|_p = \max_{Y \in B_q} \E X_tY, \qquad t \in \R^n,
\]
with the maximum attained at
\[
Y_*(t) = \frac{1}{\|X_t\|_p^{p-1}}\sgn(X_t)|X_t|^{p-1}.
\]
The main idea is to consider the subset $A_q$ of $B_q$ of random variables with nonnegative correlations with all $\e_j$,
\[
A_q = \left\{Y \in B_q, \ \E [Y\e_j] \geq 0, \ j = 1, \dots, n\right\}.
\]

\textbf{Claim.} $Y_*(t) \in A_q$, for every $t \in \R^n$.

As a result,
\[
\left\|X_t\right\|_p = \max_{Y \in A_q} \E X_tY, \qquad t \in \R^n,
\]
which allows to finish the proof in one line. We have,
\[
\frac{1}{p}\Phi(t) = \log \E\|X_t\|_p = \max_{Y \in A_q} \log\E X_tY = \max_{Y \in A_q} \log\left(\sum_{j=1}^n e^{t_j}\E [Y\e_j]\right)
\]
Functions $t \mapsto \log\sum_{j=1}^n e^{t_j}\E [Y\e_j]$ are convex (as sums of log-convex functions are log-convex), so their pointwise maximum over $Y \in A_q$ is also convex.

\begin{proof}[Proof of the claim]
Let $f(x) = \sgn(x)|x|^{p-1}$ which is nondecreasing. Fix $j \leq n$ and note that evaluating the expectation against $\e_j$ gives
\[
\|X_t\|_p^{p-1}\E [Y_*(t)\e_j] = \E[f(X_t)\e_j] = \frac{1}{2}\E\left[f\left(e^{t_j}+\sum_{i \neq j} e^{t_i}\e_i \right) - f\left(-e^{t_j}+\sum_{i \neq j} e^{t_i}\e_i \right)\right].
\]
The square bracket is nonnegative as $f(v+u) \geq f(v-u)$ for every $u \geq 0$ and $v \in \R$, by monotonicity.
\end{proof}

\begin{remark}
We have crucially used that the class of log-convex functions is closed under summation, or more generally, if $\{f_\alpha(x)\}_{\alpha \in \mathcal{A}}$ is a family of log-convex functions on, say $\R^n$, then the function 
\begin{equation}\label{eq:sum-log-convex}
x \mapsto \int_{\mathcal{A}} f_\alpha(x) \dd\mu(\alpha)
\end{equation}
is also log-convex on $\R^n$, where $\mu$ is a nonnegative measure on $\mathcal{A}$. This readily follows from H\"older's inequality. As a result, Theorem \ref{thm:radem-log-convex-p} instantly extends to sums of independent symmetric random variables (a random variable $X$ is symmetric if $-X$ and $X$ have the same distribution).
\end{remark}

\begin{corollary}\label{cor:all-sym}
Let $X_1, X_2, \dots$ be independent symmetric random variables. For every $n \geq 1$ and $p \geq 1$, the function
\[
\Phi(t_1, \dots, t_n) = \log\E\left|e^{t_1}X_1 + \dots + e^{t_n}X_n\right|^p
\]
is convex on $\R^n$.
\end{corollary}

For the proof, note that by the symmetry of the $X_j$, they have the same distribution as $\e_j|X_j|$, respectively, where $\e_1, \dots, \e_n$ are independent Rademacher random variables (independent of the $X_j$). Thus, it suffices to use \eqref{eq:sum-log-convex} with $\mu$ given by the distribution of $(|X_1|,\dots, |X_n|)$. 

In a similar vein, Corollary \ref{cor:all-sym} remains trivially true if $X_i$ are assumed to be nonnegative random variables since a sum of log-convex functions is log-convex.

\section{Concluding remarks}

\subsection{Monotonicity under $\ell_\infty$ normalisation}

Aliev in Lemma 2 in \cite{A2} showed that for the unit cube $Q_n$, its Busemann norm $N_{Q_n}(x) = \frac{|x|}{\vol_{n-1}(Q_n \cap x^\perp)}$ (see Theorem \ref{thm:Bus}) is maximised over the unit $\ell_\infty$-sphere at its vertices. Since the maximum of a convex function over a convex body is attained at an extreme point, Aliev's lemma extends in such a statement to all origin-symmetric convex bodies. Moreover, since an even convex function on the real line is nondecreasing, for $1$-symmetric bodies we obtain a stronger monotonicity property.

\begin{theorem}\label{thm:l_infty}
Let $K$ be a $1$-symmetric convex body in $\R^n$. Then the function $x \mapsto N_K(x)$ defined in \eqref{eq:def-N} is monotone with respect to each coordinate on $\R_+^n$. In particular, $\max_{x \in [0,1]^n} N_K(x) = N_K(1,\dots, 1)$.
\end{theorem}

\subsection{Dual logarithmic Brunn-Minkowski inequality}

Saroglou's dual log Brunn-Minkowski inequality, Theorem 6.2 from \cite{Sar2}, essentially states that the $(n-1)$-volume of the polytope
\[
P_t = \text{conv}\left\{\pm e^{t_j}\text{Proj}_{(1,\dots,1)^\perp}e_j, \ j \leq n\right\} = \text{Proj}_{(1,\dots,1)^\perp}\text{conv}\{\pm e^{t_j}e_j, \ j \leq n\}
\]
is log-convex. As usual, for a subspace $H$ in $\R^n$, $\text{Proj}_H$ denotes the orthogonal projection onto $H$. On the other hand, the $2^n$ facets of the stretched cross-polytope $\text{conv}\{\pm e^{t_j}e_j\}$ are all congruent with the outer normal vectors $\left(\sum_j e^{-2t_j}\right)^{-1/2}[\e_je^{-t_j}]_{j=1}^n$, $\e \in \{-1,1\}$ and the $(n-1)$-volume $n\left(\sum_j e^{-2t_j}\right)^{1/2}e^{\sum t_j}$, so from Cauchy's formula (see for instance, \cite{Gar, NT-surv}), we get
\[
\vol_{n-1}(P_t) = 2^{n-1}ne^{\sum t_j}\E\left|\sum_j e^{-t_j}\e_j\right|.
\]
Thus the convexity of the function
\begin{equation}\label{eq:convexityL1sum}
t \mapsto \log\E\left|\sum_j e^{t_j}\e_j\right|
\end{equation}
is a special case of Saroglou's result. In that sense, Theorem \ref{thm:radem-log-convex-p} can be viewed as a probabilistic extension of the dual logarithmic Brunn-Minkowski inequality. In analogy to Theorem \ref{thm:radem-log-convex-p}, we thus conjecture that the following extension of \eqref{eq:logBM-xi} holds: for every $0 < q  \leq 1$, the function
\[
t \mapsto -\log\left(\E\left[\left| e^{t_1}\xi_1 + \dots + e^{t_n}\xi_n\right|^{-q}\right]\right)
\]
is convex on $\R^n$.

\subsection{A vector-valued extension}

As an immediate corollary to Theorem \ref{thm:radem-log-convex-p}, we obtain its extension to vector-valued coefficients in Hilbert space.

\begin{corollary}\label{cor:hilbert}
Let $\e_1, \e_2, \dots$ be independent Rademacher random variables. Let $H$ be a separable Hilbert space with norm $\|\cdot\|$. Let $p \geq 1$ and $v_1, \dots, v_n$ be vectors in $H$. Then
\[
\Phi(t_1, \dots, t_n) = \log\E\left\|e^{t_1}\e_1v_1 + \dots + e^{t_n}\e_nv_n\right\|^p
\]
is convex on $\R^n$.
\end{corollary}
\begin{proof}
We use a standard embedding (see, e.g. Remark 3 in \cite{Sza}): we fix an orthonormal basis in $H$, say $(u_k)_{k \geq 1}$, take i.i.d. standard Gaussian random variables $g_1, g_2, \dots$, independent of the $\varepsilon_j$ and set $G = \sum_{k \geq 1} g_ku_k$ to have
\[
\|x\|^p = \frac{1}{\E|g_1|^p}\E\left|\scal{x}{G}\right|^p, \qquad x \in H.
\]
This gives that
\[
\Phi(t_1, \dots, t_n) = -\log \E|g_1|^p + \log\E_G\E_\e \left|\sum_{j=1}^n e^{t_j}\scal{v_j}{G}\e_j\right|^p.
\]
The result follows from Theorem \ref{thm:radem-log-convex-p}, for conditioned on the value of $G$, the function $t \mapsto \E_\e \left|\sum_{j=1}^n e^{t_j}\scal{v_j}{G}\e_j\right|^p$ is log-convex and sums of log-convex functions are log-convex.
\end{proof}

\subsection{A Representation as a maximum}

We finish with an elementary representation of the $L_1$ norm of Rademacher sums as a maximum of linear forms with nonnegative ordered coefficients. This besides being of independent interest gives an alternative proof of the convexity of  \eqref{eq:convexityL1sum}, as explained at the end of this subsection. For $n \geq 1$, we let
\[
T_n = \{x \in \R^n, \ x_1 \geq x_2 \geq \dots x_n \geq 0\}
\]
be the cone in $\R^n$ of nonincreasing nonnegative sequences.

\begin{lemma}\label{lm:rep}
For every $n \geq 1$, there is a finite subset $A_n$ of $T_n$ such that for all $x \in T_n$, we have
\[
\E\left|\sum_{j=1}^n x_j\e_j\right| = \max_{a \in A_n} \ \sum_{j=1}^n a_jx_j.
\]
\end{lemma}

\begin{proof}
Changing the order of summation, we write
\[
\E\left|\sum_{j=1}^n x_j\e_j\right| = \E\left[\sgn\left(\sum_{j=1}^n x_j\e_j\right)\left(\sum_{j=1}^n x_j\e_j\right)\right] = \sum_{j=1}^n x_j\E\left[\e_j\sgn\left(\sum_{i=1}^n x_i\e_i\right)\right],
\]
where we use the standard signum function, $\sgn(t) = |t|/t$, $t \neq 0$, $\sgn(0) = 0$ which is odd and nondecreasing. 
It is thus natural to define the function $\alpha = (\alpha_1, \dots, \alpha_n)\colon \R^n \to \R^n$,
\[
\alpha_j(x) = \E\left[\e_j\sgn\left(\sum_{i=1}^n x_i\e_i\right)\right].
\]
Note that since $\sgn(\cdot)$ is odd, we have
\[
\alpha_j(x) = \E\sgn\left(x_j + \sum_{i \neq j} x_i\e_i\right), \qquad x \in \R^n.
\]
We set
\[
A_n = \alpha(T_n)
\]
and to finish the proof, we claim that
\begin{enumerate}[(1)]
\item\label{C1} $A_n$ is a finite set,

\item\label{C2} $A_n \subset T_n$,

\item\label{C3} $\E|\sum_{j=1}^n x_j\e_j| = \max_{a \in A_n} \ \sum_{j=1}^n a_jx_j$, for every $x \in T_n$.
\end{enumerate}

Claim \eqref{C1} holds because $\alpha_j(x)$ takes only finitely many values (for any $x$, $\alpha_j(x)$ is a sum of $2^n$ terms, each equal to $\pm \frac{1}{2^n}$ or $0$).

To show \eqref{C2}, we fix $x \in T_n$ and $1 \leq k \leq n-1$. To argue that $\alpha_k(x) \geq \alpha_{k+1}(x)$, we write
\begin{align*}
\alpha_k(x) - \alpha_{k+1}(x) &= \E\sgn\left(x_k + \e_{k+1}x_{k+1} +  \sum_{i \neq k, k+1} x_i\e_i\right) - \E\sgn\left(x_{k+1} + \e_{k}x_{k} +  \sum_{i \neq k, k+1} x_i\e_i\right) \\
&= \E\left[\sgn\left(x_k -x_{k+1} +  \sum_{i \neq k, k+1} x_i\e_i\right) - \E\sgn\left(x_{k+1}  -x_{k} +  \sum_{i \neq k, k+1} x_i\e_i\right)\right]
\end{align*}
and the monotonicity of $\sgn(\cdot)$ finishes the argument. We also need to show that $\alpha_n(x) \geq 0$. Taking the expectation with respect to $\varepsilon_n$ in the definition of $\alpha_n$, we have
\[
\alpha_n(x) = \frac12\E\left[\sgn\left(x_n + \sum_{i < n} x_i\e_i\right) - \sgn\left(-x_n + \sum_{i < n} x_i\e_i\right)\right]
\]
and the expression inside the expectation is nonnegative because $\sgn(v+u) \geq \sgn(v-u)$ for every $u \geq 0$ and $v \in \R$, by monotonicity.

Finally, to prove \eqref{C3}, we fix $x \in T_n$, take arbitrary $a \in A_n$, say $a = \alpha(y)$ with $y \in T_n$ and note that
\begin{align*}
\sum_j a_jx_j = \sum_j \alpha_j(y)x_j &= \sum_j \E\left[\e_j\sgn\left(\sum_i y_i\e_i\right)\right]x_j \\
&= \E\left[\sgn\left(\sum_i y_i\e_i\right)\left(\sum_j x_j\e_j\right)\right] \\
&\leq \E\left|\sum_j x_j\e_j\right|
\end{align*}
proving that $\max_{a \in A_n} a_j x_j \leq \E|\sum x_j\e_j|$ with the equality plainly attained for $a = \alpha(x)$.
\end{proof}

If we now account for all possible orderings by taking the maximum over all permutations in the symmetric group $S_n$ on $\{1, \dots, n\}$, we obtain a representation for arbitrary coefficients.

\begin{corollary}\label{cor:full-rep}
Let $n \geq 1$ and let $A_n$ be the finite subset provided by Lemma \ref{lm:rep}. For every $x \in \R_+^n$, we have
\[
\E\left|\sum_{j=1}^n x_j\e_j\right| = \max_{a \in A_n, \sigma \in S_n} \ \sum_{j=1}^n a_jx_{\sigma(j)}.
\]
\end{corollary}

\begin{proof}
Fix $x \in \R_+^n$ and let $\sigma^*$ be a permutation such that 
\[x_{\sigma^*(1)} \geq \dots \geq x_{\sigma^*(n)}.
\]
By Lemma \ref{lm:rep},
\[
\E\left|\sum_{j=1}^n x_j\e_j\right| = \E\left|\sum_{j=1}^n x_{\sigma^*(j)}\e_j\right| = \max_{a \in A_n} \ \sum_{j=1}^n a_jx_{\sigma^*(j)}
\]
Moreover, by the rearrangement inequality, for an arbitrary permutation $\sigma$ and arbitrary $a \in A_n$ we get
\[
\sum_{j=1}^n a_jx_{\sigma^*(j)} \geq \sum_{j=1}^n a_jx_{\sigma(j)}
\]
since both sequences $(a_j)$ and $(x_{\sigma^*(j)})$ are nonincreasing. This finishes the proof.
\end{proof}

To see that the function in \eqref{eq:convexityL1sum} is convex, note that from Corollary \ref{cor:full-rep}, we have
\[
\Phi(t_1, \dots, t_n) =  \max_{\sigma \in S_n, a \in A_n} \log\sum_{j=1}^n a_je^{t_{\sigma(j)}}.
\]
For a fixed $a \in A_n$ and $\sigma \in S_n$, the function
\[
\log\sum_{j=1}^n a_je^{t_{\sigma(j)}}
\]
is convex (as sums of log-convex functions are log-convex). Thus so is their pointwise maximum.\hfill$\square$


\begin{thebibliography}{9}

\bibitem{A2}
Aliev, I.,
Siegel's lemma and sum-distinct sets.
Discrete Comput. Geom. 39 (2008), no. 1-3, 59--66.

\bibitem{A1}
Aliev, I., On the volume of hyperplane sections of a $d$-cube.
Acta Math. Hungar. 163 (2021), no. 2, 547--551.

\bibitem{Ball}
Ball, K.,
Logarithmically concave functions and sections of convex sets in $R^n$.
Studia Math. 88 (1988), no. 1, 69--84.

\bibitem{Ball-wave}
Ball, K.,
Mahler's conjecture and wavelets.
Discrete Comput. Geom. 13 (1995), no. 3-4, 271--277.

\bibitem{FB1}
B\'ar\'any, I., Frankl, P., How (not) to cut your cheese. Amer. Math. Monthly 128 (2021), no. 6, 543--552.

\bibitem{FB2}
B\'ar\'any, I., Frankl, P., Cells in the box and a hyperplane. J. Eur. Math. Soc. (JEMS) 25 (2023), no. 7, 2863--2877.

\bibitem{bhatia}
Bhatia, R., Matrix analysis. Graduate Texts in Mathematics, 169. Springer-Verlag, New York, 1997.

\bibitem{BLYZ}
B\"or\"oczky, K., Lutwak, E., Yang, D., Zhang, G.,
The log-Brunn-Minkowski inequality.
Adv. Math. 231 (2012), no. 3-4, 1974--1997.

\bibitem{Bus}
Busemann, H.,
A theorem on convex bodies of the Brunn-Minkowski type.
Proc. Nat. Acad. Sci. U.S.A. 35 (1949), 27--31.

\bibitem{Gar}
Gardner, R. J., Geometric tomography. Second edition. Encyclopedia of Mathematics and its Applications, 58.
Cambridge University Press, New York, 2006.

\bibitem{KK}
K\"onig, H., Koldobsky, A., Volumes of low-dimensional slabs and sections in the cube, Adv. Appl. Math.
47(4) (2011), 894--907.

\bibitem{MP}
Milman, V. D., Pajor, A., Isotropic position and inertia ellipsoids and zonoids of the unit ball of a normed $n$-dimensional space. Geometric aspects of functional analysis (1987--88), 64--104, Lecture Notes in Math., 1376, Springer, Berlin, 1989.

\bibitem{NT}
Nayar, P., Tkocz, T., On a convexity property of sections of the cross-polytope. Proc. Amer. Math. Soc. 148 (2020), no. 3, 1271--1278.

\bibitem{NT-surv}
Nayar, P., Tkocz, T., Extremal sections and projections of certain convex bodies: a survey. Harmonic analysis and convexity, 343--390, Adv. Anal. Geom., 9, De Gruyter, Berlin, 2023.

\bibitem{Sar2}
Saroglou, C., More on logarithmic sums of convex bodies. Mathematika 62 (2016), no. 3, 818--841.

\bibitem{Sza}
Szarek, S. J.,
On the best constants in the Khinchin inequality.
Studia Math. 58 (1976), no. 2, 197--208.




\end{thebibliography}
\end{document}